\documentclass{article}

\usepackage{geometry}                % See geometry.pdf to learn the layout options. There are lots.
\usepackage{graphicx}

\usepackage{epstopdf}

\usepackage{amsmath}
\usepackage{amssymb}
\usepackage{amsthm}
\usepackage{url}

\usepackage{pifont}% http://ctan.org/pkg/pifont
\usepackage{stmaryrd}
\usepackage{dsfont}
\usepackage{textcomp}

\newtheorem{thm}{Theorem}
\newtheorem{prop}[thm]{Proposition}
\newtheorem{lemma}[thm]{Lemma}
\newtheorem{cor}[thm]{Corollary}
\newtheorem{qu}[thm]{Question}

\theoremstyle{definition}

\newtheorem{dfn}[thm]{Definition}

\newcommand{\pow}{\mathcal P}

\newcommand{\Ord}{\mathrm{Ord}}

\newcommand{\dom}{\mathrm{dom}}

\newcommand{\stcan}{\mathrm{stcan}}

\title{Feferman's Forays into the Foundations of Category Theory}
\author{
Ali Enayat, Paul Gorbow, and Zachiri McKenzie \\
\\
	\small University of Gothenburg \\
	\small Department of Philosophy, Linguistics, and Theory of Science \\
	\small Box 200, 405 30 G\"OTEBORG, Sweden \\
	\small \url{ali.enayat@gu.se}, \url{paul.gorbow@gu.se}, \url{zachiri.mckenzie@gu.se} \\
}
\begin{document}
\maketitle

\section{Introduction}

The foundations of category theory has been a source of many perplexities ever
since the groundbreaking 1945-introduction
of the subject by Eilenberg and Mac Lane; e.g.,  how is one to avoid Russell-like paradoxes and yet have
access to  objects that motivate the study in the first place, such as the category of all groups, or the  category of all topological spaces?
Solomon Feferman has grappled with such perplexities for over 45 years, as
witnessed by his six papers on the subject during the period
1969-2013 \cite{Fef69, Fef74, Fef77, Fef04, Fef06, Fef13}. Our focus in this paper is on two important, yet quite different set-theoretical systems
proposed by Feferman for the implementation of category theory: the $\mathrm{ZF}$-style system $\mathrm{ZFC/S}$ \cite{Fef69} and the $\mathrm{NFU}$-style
system $S^{* }$ \cite{Fef74, Fef06}; where $\mathrm{NFU}$ is Jensen's
urelemente-modification of Quine's New Foundations system $\mathrm{NF}$ of
set theory.\footnote { Feferman \cite{ Fef77, Fef13} has put forward poignant criticisms of the general case of using category theory as an autonomous foundation for mathematics. Moreover, he suggested that a theory of {\em operations and collections} should also be pursued as a viable alternative platform for category theory; e.g. systems of Explicit Mathematics \cite{Fef75} and Operational Set Theory \cite{Fef09}.}

Our assessment of Feferman's systems $\mathrm{ZFC/S}$
and $S^{\ast }$ will be framed by the following general desiderata (R) and (S) of
the whole enterprise of building a set-theoretical foundation for category theory. (R) is derived from
Feferman's work, especially \cite{Fef74, Fef06}, while (S) is extracted from
Shulman's excellent survey \cite{Shu08}.  Both (R) and (S) will be elaborated in Subsection 2.1.

\begin{enumerate}
\item[(R)] asks for the \em unrestricted existence \em of the category of all groups, the category of all categories, the category of all functors between two categories, etc., along with \em natural implementability of ordinary mathematics and category theory. \em
\item[(S)] asks for a certain \em relative distinction between large and small sets, \em and the requirement that they \em both enjoy the full benefits of the $\mathrm{ZFC}$  axioms. \em
\end{enumerate}

Feferman's choice for a system to meet the demands of (R) was motivated by the fact that in contrast with $\mathrm{ZFC}$-style systems, $\mathrm{NFU}$ accommodates a universal set of all sets, as well as the category of all groups, the category of all categories, etc. However, in order to deal with the fact that $\mathrm{NFU}$  is not powerful enough to handle some parts of ordinary mathematics and category theory, Feferman proposed an extension of $\mathrm{NFU}$, called $S^*$, which directly interprets $\mathrm{ZFC}$ by a constant, and he established the consistency of $S^*$  assuming the consistency of  $\mathrm{ZFC} + \exists \kappa \exists \lambda$ ``$\kappa < \lambda$ are inaccessible cardinals'' \cite{Fef74}. In our expository account of Feferman's work on  $S^*$, we refine Feferman's results in two ways: we show that $S^*$ interprets  a significant strengthening of the Kelley-Morse theory of classes; and we also demonstrate the consistency of a natural strengthening $S^{**}$ of $S^*$ within $\mathrm{ZFC} + \exists \kappa$ ``$\kappa$ is an inaccessible cardinal''.

A widely accepted partial (partial because it is absolute, not relative) solution to (S) was chosen by Saunders Mac Lane in his standard reference ``Categories for the Working Mathematician'' \cite{Mac98}, namely the theory $\mathrm{ZFC}$  + $\exists \kappa$ ``$\kappa$ is an inaccessible cardinal''. Here the sets of rank less than a fixed inaccessible cardinal are regarded as small. A similar (but full, i.e. relative) solution to (S) due to Alexander Grothendieck, requires the existence of arbitrarily large inaccessible cardinals. However, as demonstrated by Feferman  \cite{Fef69}, the needs of category theory in relation to (S) can already be met in a conservative extension  $\mathrm{ZFC/S}$ of $\mathrm{ZFC}$, thus fully satisfying (S), albeit with the consequence that, while working with categories and functors in this foundational system, one often needs to verify that they satisfy a property called small-definability.

Besides the aforementioned fine-tuning of Feferman's results on $S^*$, our paper also has two other innovative features that were inspired by Feferman's work. One has to do with the observation that the category $\mathbf{Rel}$ in $\mathrm{NFU}$ has products and coproducts indexed by $\{\{i\} \mid i \in V\}$ (the same existence result also holds for coproducts in $\mathbf{Set}$). The importance lies in that it marks a divergence from the analogous category of small sets with relations (respectively functions) as morphisms in a $\mathrm{ZFC}$ setting; this situation in $\mathrm{ZFC}$ is clarified by a theorem of Peter Freyd, which we prove in detail (Freyd's theorem also serves to exemplify why the large/small distinction is important for category theory). We consider this positive result on the existence of some limits for $\mathbf{Rel}$ and $\mathbf{Set}$ in NFU, to strengthen the motivation for exploring NFU-based foundations of category theory. Thus, another innovative feature is our proposal to champion the system $\mathrm{NFUA}$ (a natural extension of $\mathrm{NFU}$) in meeting the demands imposed by both (R) and (S).  It is known that $\mathrm{NFUA}$ is equiconsistent with $\mathrm{ZFC}$ + the schema: ``for each natural number $n$ there is an $n$-Mahlo cardinal''.\footnote{This equiconsistency result is due to Robert Solovay, whose 1995-proof is unpublished. The first-named-author used a different proof in \cite{Ena04} to establish a refinement of Solovay's equiconsistency result.} This puts $\mathrm{NFUA}$ in a remarkably close relationship to the system $\mathrm{ZMC/S}$, which is a strengthening of $\mathrm{ZFC/S}$ suggested by Shulman \cite{Shu08} as a complete solution to (S). 

\section{Requirements on foundations for category theory}

\subsection{The problem formulations (R) and (S)} \label{FormulationsRandS}

We will now elaborate each of the desiderata (R) and (S) mentioned in the introduction by more specific demands. The following requirements have been obtained from Feferman \cite{Fef74, Fef06}.

\begin{enumerate}
\item[(R1)] For any given kind of mathematical structure, such as that of groups or topological spaces, the associated category of all such structures exists.
\item[(R2)] For any given categories $A$ and $B$, the category $B^A$ of all functors from $A$ to $B$ with natural transformations as morphisms, exists. 
\item[(R3)] Ordinary mathematics and category theory, along with its distinction between large and small, are naturally implementable.
\end{enumerate}

From a purely formal point of view, a foundational system is only required to facilitate implementation of the mathematical theories it is founding. But in practice, it is also desirable for this implementation to be user friendly. This is what the informal notion ``naturally implementable'' above seeks to capture. For the purposes of this paper, it suffices to distinguish three levels of decreasing user friendliness within the admittedly vague class of natural implementations of category theory in set-theory: (1) The usual set-theoretical description of category theoretic notions can be accommodated directly with the membership relation symbol of the underlying language, and with no restriction. (2) As above, but restricted to a set or class. (3) The notions can be accommodated with another well motivated defined membership relation, and restricted to a set or class. For example, the system NFUA facilitates a level (3) implementation of category theory by means of the set of equivalence classes of pointed extensional well-founded structures. But we will see that an equiconsistent extension of NFUA facilitates a level (2) implementation.

Feferman actually denotes his requirements as ``(R1)-(R4)''. His (R1)-(R2) are as above. His (R4) asks for a consistency proof of the system relative to some standard system of set theory; which we take as an implicit requirement. His (R3) requires that the system allows ``us to establish the existence of the usual basic mathematical structures and carry out the usual set-theoretical operations'' \cite{Fef06}. Feferman proposed a system $S^*$, which will be looked at in detail in the present paper, as a partial solution to his problem (R). But at least on one natural reading of Feferman's (R), Michael Ernst has recently shown in \cite{Ern15} that Feferman's (R3) is inconsistent with (R1) and (R2). Ernst uses a theorem of William Lawvere, which formulates Cantor's diagonalization technique in a general category theoretic context \cite{Law69}. Although Ernst argues generally for the inconsistency of Feferman's (R), it is worth noting that there is a more direct result that applies to NFU: Feferman's (R3) may be understood to require that the category $\mathbf{Set}$ of sets and functions is cartesian closed, but two different proofs of the failure of cartesian closedness for $\mathbf{Set}$ in NF (also applicable to NFU) are provided in \cite{McL92} and \cite{For07}.

Our formulation of (R3) is weaker in the sense that we only require {\em natural implementability}. We will show that Feferman's system $S^*$ (as well as NFUA), solve this version of (R).

Shulman \cite{Shu08} does not provide an explicit list of requirements, but we extract the following from his exposition.

\begin{enumerate}
\item[(S1)] Ordinary category theory, along with its distinction between large and small, is naturally implementable.
\item[(S2)] The large/small-distinction is relative, in the sense that for any $x$, there is a notion of smallness such that $x$ is small.
\item[(S3)] ZFC is interpreted by $\in$, both when quantifiers are restricted to large sets, and when quantifiers are restricted to small sets.
\end{enumerate}

In order to avoid confusion, let us just clarify that Shulman does not directly advocate that a foundation of category ought to satisfy all of these conditions. Rather, in his concluding section he leaves that ``to the reader's aesthetic and mathematical judgement'' \cite{Shu08}. We have just extracted these conditions from the motivations Shulman gives for adopting one or another of the many ZFC-based foundations for category theory he considers. On our reading, Shulman's story-line for ZFC-based approaches culminates with systems aimed at fully satisfying (S).

\subsection{The distinction between large and small}

In ZFC, the familiar categories $\mathbf{Set}, \mathbf{Top}, \mathbf{Grp}$ can only be treated as proper classes. Still, it is perfectly possible to state and prove some theorems about a particular category, e.g. ``for any two groups, there is a group, which is their product''. However, if we want to prove that something holds for every category that satisfies some condition, say ``for every category with binary products, we have that \dots", and if we want this result to be applicable to $\mathbf{Set}, \mathbf{Top}, \mathbf{Grp}, \dots$, then ZFC is insufficient.

Several extensions of ZFC have been proposed for dealing with this situation, that are concerned with a distinction between large and small sets. In the standard reference on category theory by Saunders Mac Lane, ``Categories for the Working Mathematician'' \cite{Mac98}, ZFC + $\exists \kappa$ ``$\kappa$ is an inaccessible cardinal'', is chosen as the foundational system. The sets of rank less than $\kappa$ are considered small, and all sets are considered large. The set of small sets, $V_\kappa$, is then a model of ZFC. It is also common to say that $V_\kappa$ is a Grothendieck universe, but we will not use that terminology in the present paper. The point is that any mathematical theory that can be developed within ZFC, can be developed exclusively by means of small sets. In ZFC + $\exists \kappa$ ``$\kappa$ is an inaccessible cardinal'', we can easily construct $\mathbf{Set}, \mathbf{Top}, \mathbf{Grp}, \dots$ as the categories of all small sets, small groups, small topological spaces, \dots (these categories will not be small however). In fact, we can conveniently state and prove many things about categories, as witnessed by Mac Lane's comprehensive book.

One might suppose that the category theory unfolding from this approach would only utilize the notion of smallness to obtain a formal foundation for stating and proving results. But, as Shulman observes \cite{Shu08}, many results of category theory are actually concerned with this notion in a mathematically interesting way: For a simple example, we will prove in detail that (in such a ZFC-based approach) any small (large) category which has small (large) limits, is just a preorder. As a consequence, it is more interesting to study the situation of a large category having small limits, a condition satisfied by $\mathbf{Set}, \mathbf{Top}, \mathbf{Grp}$.  Important well-known examples making essential use of smallness are the Yoneda lemma and the adjoint functor theorems.

\section{A quick dip into category theory}

\subsection{Limits}

We have definitions of groups and group homomorphisms, topological spaces and continuous functions, and so on for many types of objects and morphisms in many areas of mathematics. One of the ways in which category theory is helpful, is that it proceeds abstractly from these basic data of objects and morphisms, to define standard notions like subobjects, quotients and products, for all categories at once. But category theory does not only provide guidelines for defining familiar notions in new categories. It also guides us when introducing new notions (beyond the familiar product and so on) in familiar categories. As an example, we will consider the category theoretic notion of limit, of which the notion of product is just one of the simplest instances.

Throughout the paper, we will occasionally use ``$X \in \mathbf{C}$'', as shorthand for ``$X$ is an object of the category $\mathbf{C}$''. Suppose that $\mathbf F : \mathbf J \rightarrow \mathbf C$ is a functor. A \em cone \em to $\mathbf F$ is an object $N$ of $\mathbf C$ together with a set of morphisms $\{n_X : N \rightarrow \mathbf F(X) \mid X \in \mathbf J\}$, such that for any $f: X \rightarrow_\mathbf J Y$, the following diagram commutes.

\begin{center}
\includegraphics{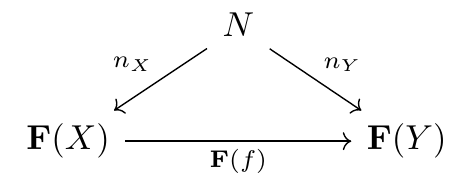}
\end{center}

A \em limit \em to $\mathbf F$ is a universal cone, i.e. a cone $U, \{u_X\}$ to $\mathbf F$, such that for any cone $N, \{n_X\}$ to $\mathbf F$, there is a unique $v : N \rightarrow_\mathbf C U$, such that for any $f: X \rightarrow_\mathbf J Y$, the following diagram commutes.

\begin{center}
\includegraphics{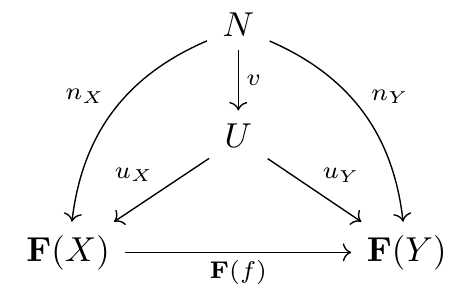}
\end{center}

Since the morphism $v$ (if it exists) is uniquely determined by the cone $N \dots$ and the limit $U \dots$, we call it the \em universal morphism from the cone $N \dots$ to the limit $U \dots$\em. Moreover, $\mathbf J$ is referred to as an \em index \em category, and $\mathbf F : \mathbf J \rightarrow \mathbf C$ is called a \em diagram\em. Of course, an index category can be any category, and a diagram can be any functor. We say that $\mathbf C$ has \em limits indexed by $\mathbf J$\em, if for each diagram $\mathbf F : \mathbf J \rightarrow \mathbf C$, it has a limit to $\mathbf F$. 

Small index categories and diagrams are of particular importance. We say that $\mathbf{C}$ has \em small limits\em, if for any small index category $\mathbf J$ and any small diagram $\mathbf F : \mathbf J \rightarrow \mathbf C$, we have a limit in $\mathbf C$ to $\mathbf F$. 

If the category $\mathbf C$ is a preorder, i.e. a category with at most one morphism between any two objects, and $a \geq b$ corresponding to an arrow $a \rightarrow b$, then a limit to $\mathbf F$ is the same as a least upper bound of the image of the object function of $\mathbf F$ (so the morphisms in $\mathbf J$ will be inconsequential). {\em Hence, the notion of the existence of limits for categories generalizes the notion of completeness for preorders.} We will see further below that with a ZFC-based foundation, every small category with small limits is a preorder. However, many large non-preordered categories, such as $\mathbf{Set}, \mathbf{Top}, \mathbf{Grp}$, have small limits.

We now show how a product is an example of a limit. Let $\mathbf 2$ be the category consisting of only two objects, say $0$ and $1$, along with their identity morphisms. Suppose $\mathbf C$ is any category, and $A, B \in \mathbf C$. Then the map sending $0$ to $A$ and $1$ to $B$ determines a functor $\mathbf F : \mathbf 2 \rightarrow \mathbf C$. A \em product \em of $A$ and $B$ is (defined as) a limit to $\mathbf F$. We may familiarly write this limit as $A \times B$, $\pi_0 : A \times B \rightarrow A$, $\pi_1 : A \times B \rightarrow B$, where $\pi_0, \pi_1$ are called the \em projection \em morphisms. The universality property now spells out, that for any object $N \in \mathbf C$ and any morphisms $f : N \rightarrow A$ and $g : N \rightarrow B$ (i.e. for any cone to $\mathbf F$), there is a unique morphism $\langle f, g \rangle : N \rightarrow A \times B$, such that the following diagram commutes.

\begin{center}
\includegraphics{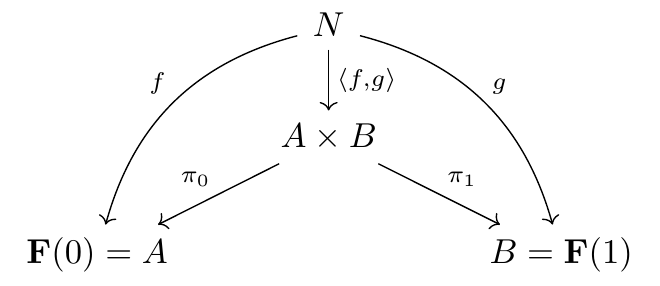}
\end{center}

If $\mathbf C = \mathbf{Set}$, then the standard cartesian product of $A, B$ together with the standard projection functions, is a limit to $\mathbf F$. Moreover, we would have $\langle f, g \rangle (x) = \langle f(x), g(x) \rangle$, for each $x \in N$, explaining why the ordered pair notation $\langle f, g \rangle$ is commonly abused for the purpose of denoting that morphism.

Of course, such a limit may not exist. But if, for every functor $\mathbf F : \mathbf 2 \rightarrow \mathbf C$, a limit to $\mathbf F$ exists, then we say that $\mathbf C$ has \em binary products\em. In this sense, given the notion of limit, we may think of the index category $\mathbf 2$ as providing all of the data for the notion of binary products. Similarly, a limit to a diagram from the index category with two objects and two morphisms $A \rightrightarrows B$ (in addition to the identity morphisms) is an {\em equalizer}; and a limit to a diagram from the index category with three objects and two morphisms $A \rightarrow C$, $B \rightarrow C$ (in addition to the identity morphisms) is a {\em pullback}. Thus, we may informally think of the notion of limit as defining a map from index categories to notions. The notions of product, equalizer and pullback would be three particular values of this map.

\subsection{The limits of a proper category (in ZFC)}

Recall that the category $\mathbf 2$ used to define binary products, only has identity morphisms. By considering index categories only having identity morphisms, more generally, we obtain stronger notions of product.

\begin{thm} \label{FreydLimits}
Assume that our set-theoretic foundation for category theory satisfies ZFC. Suppose that $\mathbf C$ is a category, and let $\mathrm{Arr}(\mathbf{C})$ be the category with the arrows of $\mathbf C$ as objects and only identity morphisms. If $\mathbf C$ has limits (i.e. products) indexed by $\mathrm{Arr}(\mathbf C)$, then $\mathbf C$ is a preorder.
\end{thm}
\begin{proof}
Assume towards a contradiction, that we have two different morphisms $f, g : A \rightarrow_\mathbf C B$. Let $P, \{p_X \mid X \in \mathrm{Arr}(\mathbf C)\}$ be a limit (product) to the constant diagram mapping every object of $\mathrm{Arr}(\mathbf C)$ to $B$. 
In suggestive notation, we may write $P = \Pi_{X \in \mathrm{Arr}(\mathbf C)} B$. 
By the universality property of limits, we now obtain a function  $\mathcal U : 2^{\mathrm{Arr}(\mathbf C)} \rightarrow \mathbf{C}(A, P)$, defined by sending each function $S : \mathrm{Arr}(\mathbf C) \rightarrow 2$ to the universal morphism from the cone 
\[A, \big\{a^S_X \mid X \in \mathrm{Arr}(\mathbf{C})\} \text{, where } \forall X \in \mathrm{Arr}(\mathbf C) . [(a^S_X = f \leftrightarrow S(X) = 0) \wedge (a^S_X = g \leftrightarrow S(X) = 1)] ,\]
to the limit $P, \{\pi_X \mid X \in \mathrm{Arr}(\mathbf C)\}$.

Since $\mathbf C(A, P)$ is a subset of (the objects of) $\mathrm{Arr}(\mathbf C)$, it suffices to prove that $\mathcal U$ is injective, for then Cantor's theorem yields a contradiction. So suppose that $S, S' : \mathrm{Arr}(\mathbf C) \rightarrow 2$ are different. Then there is $X \in \mathrm{Arr}(\mathbf C)$, such that $S(X) \neq S'(X)$. Thus, we may assume that $a^S_X = f$ and $a^{S'}_X = g$. Since $P \dots$ is a limit, the following diagram commutes.

\begin{center}
\includegraphics{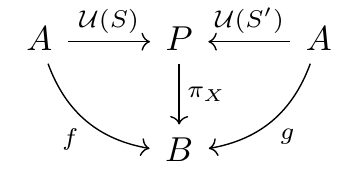}
\end{center}

Finally, using the assumption $f \neq g$, we conclude that $\mathcal U(S) \neq \mathcal U(S')$.
\end{proof}

\begin{cor}
Assume that our set-theoretic foundation satisfies ZFC. Then every small category, which has small limits, is a preorder; and every large category, which has large limits, is a preorder.
\end{cor}
\begin{proof}
A small category has a small set of morphisms, and a large category has a large set of morphisms.
\end{proof}

So in terms of large and small, the strongest notion of ``having all limits'', that may still be interesting for a category theory based on a foundation satisfying ZFC, is that of a large category having small limits. Therefore, some authors plainly call such a category \em complete\em. It turns out that many of the most familiar categories are complete, including $\mathbf{Set, Top, Grp}$. The category of small metric spaces and the category of small fields are examples of large categories which are not complete.

\section{Approaches to satisfying (S)}

This section is a brief summary of sections 8-11 in Shulman's exposition \cite{Shu08}, where various ZFC-based approaches to the foundations of category theory are assessed and compared in detail.

\subsection{Approaches based on inaccessible cardinals} \label{ApproachesInaccessibles}

We have already mentioned the system ZFC + $\exists \kappa$ ``$\kappa$ is an inaccessible cardinal'', which seems to work for most applications. We take Mac Lane's book \cite{Mac98} as sufficient evidence that it satisfies (S1). It fully satisfies (S3): Firstly, the set of small sets, $V_\kappa$, is a model of ZFC. Secondly, since every set is large, and ZFC is a subset of the theory, the large sets satisfy ZFC. This is an important difference from the system KMC (Kelley-Morse set theory with choice) -- which has also been suggested as a foundation for category theory -- where the classes (considered as large sets) do not satisfy ZFC.\footnote{
$\mathrm{KM} + \Pi^1_\infty$-AC is bi-interpretable with $T := \mathrm{ZFC} - \text{Power Set} + \exists \kappa$ ``$\kappa$ is an inaccessible cardinal, and $\forall x \; |x| \leq \kappa$'', where $\Pi^1_\infty$-AC is the schema of Choice whose instances are of the form $\forall s \: \exists X \: \phi(s, X) \rightarrow \exists Y \: \forall s \: \phi(s, (Y)_s)$, where $\phi(s, X)$ is a formula of class theory with set variable $s$ and class variable $Y$, and $(Y)_s$ is the ``$s$-th slice of $Y$'', i.e., $(Y)_s = \{x \mid \langle s, x \rangle \in Y\}$. This bi-interpretability was first noted by Mostowski; a modern account is given in a recent paper of Antos and Friedman \cite[§2]{AF15}, where $\mathrm{KM} + \Pi^1_\infty$-AC is referred to as $\mathrm{MK}^*$, and $T$ is referred to as $\mathrm{SetMK}^*$.
}

But ZFC + $\exists \kappa$ ``$\kappa$ is an inaccessible cardinal'' imposes a rigid notion of smallness, fixed once and for all. Hence, it does not satisfy requirement (S2). The results of category theory involving the notion of smallness, do not depend on any fixed collection of objects. Rather, it is the relative smallness that is important. This is already evident in the proposal ZFC  + $\exists \kappa$ ``$\kappa$ is an inaccessible cardinal'': It is not important to choose, say the least inaccessible, as the demarcation of smallness, but only to pick some inaccessible. 

Grothendieck proposed a stronger system, ZFC + ``there are arbitrarily large inaccessible cardinals'', to the following effect. For any set $A$, there is a notion of smallness such that $A$ is small, and the set of small sets is a model of ZFC: Pick an inaccessible $\kappa$ larger than the rank of $A$, and stipulate that a set is small if and only if it is an element of $V_\kappa$. Thus, this theory, which is known as Tarski-Grothendieck set theory (or just TG), satisfies (S2). It also satisfies (S1) and (S3) for the same reasons as given for ZFC + $\exists \kappa$ inaccessible above.

A worry with Mac Lane's approach, which is even more worrisome in Grothendiek's approach, is that we increase the consistency strength of the theory. The motivations for passing to these theories have had more to do with obtaining a useful notion of smallness, than with a pressing need from category theory for more consistency strength, so one might wonder if it is really necessary to introduce such large cardinal axioms. There are at least two concerns with increasing the consistency strength of the foundational system.

\begin{enumerate}
\item The risk that the theory is inconsistent increases.
\item The theory becomes inconsistent with other axioms that may turn out to be of interest to category theory.
\end{enumerate}

As for the first concern, the large cardinal axioms we have looked at so far, are much weaker than other ones, which have also been quite thoroughly studied by set theorists. Hence, the risk that e.g. TG is inconsistent, is thought to be small.

As for the second concern, the authors are not aware of any principles that are both attractive to category theory and consistent with ZFC, but inconsistent with ZFC + a large cardinal axiom. However, we will of course consider the principle (R1) in this paper, which is attractive to category theory but inconsistent with ZFC.

\subsection{Feferman's ZFC/S as a solution to (S)} \label{FefermanZFCoverS}

Now, even though we do not have strong reasons against including large cardinal axioms in a foundational system for category theory, a solution to (S) which is conservative over ZFC would yield insight on the large/small distinction, and is therefore of interest to the foundations of category theory. So how might we satisfy (S), while remaining at the consistency strength of ZFC? This is essentially the question which Feferman answers with his system ZFC/S, ``ZFC with smallness''. The idea of ZFC/S is to utilize the Reflection principle (provable in ZFC). 

\begin{thm}[Reflection principle in ZFC]
If $\Phi$ is a finite set of formulas, and $A$ is a set, then there is a $V_\alpha \supset A$, such that for any $\phi(x_0, \dots, x_{n-1}) \in \Phi$, and any $a_0, \dots, a_{n-1} \in V_\alpha$,
\[ \phi^{V_\alpha}(a_0, \dots, a_{n-1}) \leftrightarrow \phi(a_0, \dots, a_{n-1}) .\]
\end{thm}

$V_\alpha$ may be understood as the model obtained in ZFC by restricting $\in$ to the set $V_\alpha$. If $\hat\phi$ is a name for $\phi$ in ZFC, then the formula $V_\alpha \models \hat\phi$, is defined recursively in the familiar Tarskian way of the semantics of first order logic. It naturally turns out that it is equivalent to the formula $\phi^{V_\alpha}$, which is obtained from $\phi$ by restricting each quantifier to $V_\alpha$. We say that the model $V_\alpha$ \em reflects \em the formulas $\Phi$. Hence, we implicitly assume that $\Phi$ consists of formulas in the language of set theory. 

The system ZFC/S is defined as follows. Add an extra constant symbol $\mathbb S$ to the language of set theory. ZFC/S is the theory in this language consisting of the axioms of ZFC, the axiom saying that $\mathbb S$ is transitive and supertransitive, i.e. the axioms
\[ x \in y \in \mathbb S \rightarrow x \in \mathbb S \] 
\[ x \subset y \in \mathbb S \rightarrow x \in \mathbb S ,\]
plus the axiom schema that, for each natural number $n$ and each formula $\phi(x_0, \dots, x_{n-1})$ in the language of set theory, adds the axiom
\[ \forall x_0 \in \mathbb S \dots \forall x_{n-1} \in \mathbb S . \big[ \phi^\mathbb S(x_0, \dots, x_{n-1}) \leftrightarrow \phi(x_0, \dots, x_{n-1}) \big].\]

Working in $\mathrm{ZFC/S}$, we interpret ``small'' as ``$\in \mathbb{S}$'' (every set is considered large). By the Reflection schema, we may interpret $\mathrm{ZFC}$ by restricting quantifiers to $\mathbb S$, so (S3) is satisfied.

\begin{thm} \label{ConZFCoverS}
If $\mathrm{ZFC}$ is consistent, then $\mathrm{ZFC/S}$ is consistent.
\end{thm}
\begin{proof}
Suppose that $\lightning_{\mathrm{ZFC/S}}$ is a proof of $\Psi \vdash \bot$, for some finite $\Psi \subset \mathrm{ZFC/S}$. Note that $\Psi = \Psi_\mathrm{ZFC} \cup \Psi_\mathrm{Trans} \cup \Psi_\mathrm{Refl}$, where $\Psi_\mathrm{ZFC} \subset \mathrm{ZFC}$, $\Psi_\mathrm{Trans} \subset \{\textrm{Transitivity, Supertransitivity}\}$ and $\Psi_\mathrm{Refl} \subset \textrm{Reflection schema}$. Each $\psi \in \Psi_\mathrm{refl}$ is of the form $\forall x \in \mathbb S . \big( \phi^\mathbb S (\vec{x}) \leftrightarrow \phi(\vec{x}) \big)$. By reflection, 
\[
\mathrm{ZFC} \vdash \exists V_\alpha . \bigwedge_{\psi \in \Psi_\mathrm{Refl}} \psi[\mathbb S := V_\alpha].
\label{refl}
\tag{\dag}
\]
Pick a witness $V_\alpha$ of (\ref{refl}). Then $\mathrm{ZFC} \vdash \Psi_\mathrm{Refl}[\mathbb S := V_\alpha]$ is just (\ref{refl}). $\mathrm{ZFC} \vdash \Psi_\mathrm{ZFC}[\mathbb S := V_\alpha]$ is trivial, and $\mathrm{ZFC} \vdash \Psi_\mathrm{Trans}[\mathbb S := V_\alpha]$ is basic. So there is a proof $p$ of $\mathrm{ZFC} \vdash \Psi[\mathbb S := V_\alpha]$. Concatenating $p$ with $\lightning_{\mathrm{ZFC/S}}[\mathbb S := V_\alpha]$ yields a proof $\lightning_{\mathrm{ZFC}}$ of $\mathrm{ZFC} \vdash \bot$.
\end{proof}

The technique used in the proof can be applied quite generally: Suppose we are building a proof from ZFC, in which we want to apply a finite set of theorems $T$ of ZFC/S, and in which we would like to consider all the elements of some set $A$ to be small. As in the proof of $\mathrm{Con(ZFC)} \rightarrow \mathrm{Con(ZFC/S)}$, the Reflection principle gives us 
\[
\mathrm{ZFC} \vdash \exists V_\alpha \supset A . \bigwedge_{\tau \in T} \tau[\mathbb S := V_\alpha].
\label{refla}
\tag{\ddag}
\] 

Thus, given any finite set of theorems $T$ of $\mathrm{ZFC/S}$ and any set $A$, all of whose elements we want to consider small, we can choose $V_\alpha \supset A$ and translate each theorem $\tau \in T$ of $\mathrm{ZFC/S}$ into the theorem $\tau[\mathbb S := V_\alpha]$ of $\mathrm{ZFC}$. Then we can apply the translated versions in a proof from $\mathrm{ZFC}$. 

On the other hand, suppose $T$ is an infinite theorem schema in $\mathrm{ZFC/S}$. By a generalized version of the Reflection principle, the same technique works if $T$ is $\Sigma_n$ (in Levy's hierarchy of set-theoretic formulae) for some finite $n$, but it \em does not work \em for an arbitrary $T$.

We conclude that $\mathrm{ZFC/S}$ gives us a relative large/small-distinction, so that (S2) is satisfied, albeit with the limitation on $T$ stated above.

What about (S1)? Can standard results of category theory be obtained in $\mathrm{ZFC/S}$? The evidence `on the ground' seems to indicate a positive answer; e.g., Andreas Blass has informed the authors that he used $\mathrm{ZFC/S}$ as a foundational platform in a course he taught on category theory at the University of Michigan, and that everything worked smoothly, except possibly for Kan extensions. The main issue with using $\mathrm{ZFC/S}$ in practice, seems to be that we often need to ensure that categories, functors, etc are ``small-definable'', i.e. that they can be defined by a formula of $\mathrm{ZFC/S}$ all of whose parameters are elements of $\mathbb S$. For example, consider restricting a properly large locally small category $\mathbf C$ to a small subset $A$ of the objects of $\mathbf{C}$.  If $\mathbf{C}$ is not small-definable, then the replacement axiom on $\mathbb S$ (obtained from the corresponding instance of the Reflection schema) cannot be applied. Therefore, we cannot conclude that the restriction of $\mathbf{C}$ to $A$ is small. But categories, functors, etc that are definable as classes over $\mathrm{ZFC}$, obviously have direct counterparts that are small-definable in ZFC/S. Therefore, in applications this condition is almost always satisfied. 

We have now justified the claim that $\mathrm{ZFC/S}$ is a conservative extension of $\mathrm{ZFC}$ satisfying (S), albeit with minor limitations.

\section{Approaches to satisfying (R)}

\subsection{A very short introduction to stratified set theory}

(R1) asks for the category of all sets, the category of all groups, the category of all categories, etc. A natural place to look for a solution to (R), is therefore set theories proving the existence of a universal set. An approach to a set theory with this property is the simple theory of types, developed by Bertrand Russell and Alfred North Whitehead, and simplified by Leon Chwistek and Frank Ramsey, independently. In the usual form, called TST, the theory has types indexed by $\mathbb N$. There are countably infinitely many variables of each type, and an element- and equality-symbol for each type. A formula is well-formed only if each atomic subformula is of the form $x^i =^i y^i$ or $x^i \in^i y^{i+1}$, where $i \in \mathbb{N}$ denotes a type. For each type the theory has a corresponding axiom of extensionality and an axiom schema of comprehension. Russell's paradox is avoided because the formula `$x \in y$' becomes well-formed only if $x$ is assigned one type lower than $y$, thus banishing any formula of the form `$x \in x$' from the language. But `$x^i =^i x^i$' is well-formed, so by comprehension there is a universal set of type $i+1$ for each type $i \in \mathbb{N}$. 

It is easily seen that $\mathrm{TST} \vdash \phi \leftrightarrow \mathrm{TST} \vdash \phi^+$, where $\phi^+$ is obtained by raising all the type indices in $\phi$ by $1$: A proof of $\phi$ can be turned into a proof of $\phi^+$ simply by raising all the type indices in the proof by $1$, and vice versa. From reflection on this fact, Quine was led to suggest the theory NF \cite{Qui37}, which is obtained from TST simply by forgetting about the types. The result is an untyped first-order theory with the axiom of extensionality and the the axiom schema of stratified comprehension. A formula is stratified if its variables can be assigned types so as to become a well-formed formula of TST. The stratified comprehension schema only justifies formation of sets which are the extensions of stratified formulas. But the well-formed formulas of NF are as usual in set theory. For example, in contrast to TST, the statement $\mathrm{NF} \vdash \exists x . x \in x$ makes sense, and is in fact witnessed true by the universal set $V$ obtained from the instance $\exists y . \forall x .(x \in y \leftrightarrow x = x)$ of the stratified comprehension axiom schema. 

It has turned out to be very difficult to prove the consistency of NF relative to a set theory in the Zermelo-Fraenkel tradition.\footnote{Two claimed proofs have fairly recently been announced: first by Randall Holmes \cite{Hol15}, and then by Murdoch Gabbay \cite{Gab14}.} However, if the extensionality axiom is weakened to allow for atoms, then the system NFU is obtained. This system was proved consistent relative to Mac Lane set theory\footnote{The axioms of Mac Lane set theory are specified in Section 6.} in 1969 by Ronald Jensen \cite{Jen69}. Jensen's proof combines two techniques due to Ernst Specker and Frank Ramsey respectively. Specker had showed essentially that NF is equiconsistent with TST plus an automorphism between the types, and Jensen used Ramsey's theorem to obtain the automorphism at the cost of weakened extensionality. Jensen's proof also gives the consistency of NFU plus the axiom of choice and the axiom of infinity. This system is equiconsistent with NFU + the axiom of choice + the axiom of type-level pair. It is standard in the literature to refer to this latter system simply as NFU, and we do the same in this paper. A thorough introduction to NFU is found here \cite{Hol98}.

The move to allow for atoms in NFU can be accomplished either by having extensionality only for non-empty sets, or by introducing a predicate of sethood and restricting extensionality to sets. We opt for the former alternative. It is then helpful to have a designated empty set; i.e. we pick an atom once and for all, and denote it by $\varnothing$. (The existence of atoms, now defined simply as sets without elements, is provided by applying stratified comprehension to the formula $x \neq x$.)

\subsection{Feferman's $S^*$ as a solution to (R)} \label{Sec:FefSstar}

Jensen's proof of the consistency of NFU appeared in 1969; it did not take long for Feferman to use NFU, in the early 1970's, as the source of a viable solution to (R). Indeed, it is easy to see that NFU satisfies (R1) and (R2). A detailed verification is found in \cite[§3.4]{Fef74}. But since NFU is equiconsistent with Mac Lane set theory, which is proved consistent in ZFC, it follows from G\"odel's second incompleteness theorem that NFU does not interpret ZFC. Let us now consider (R3). If ZFC is taken as the standard for ``ordinary mathematics'', then it is clear that NFU does not meet that standard. However, the equiconsistency of NFU with Mac Lane set theory, which arguably suffices for ordinary mathematics and category theory, suggests a path for advocating NFU as a foundation for category theory.

Feferman suggests the system $S^*$, which extends NFU with a constant symbol $\bar V$, and axioms ensuring that $\langle \bar V, \in \rangle$ is a model of ZFC (actually with stronger replacement and foundation schemata than required for that) \cite{Fef74, Fef06}.

The language of $S^*$, denoted $\mathcal{L}^*$, is the two-sorted extension of the language of set theory with set variables $x, y, z, \ldots$, class variables $X, Y, Z, \ldots$, a constant symbol $\bar V$ and a binary function symbol $P$. As we alluded to above, the function symbol $P$, which will act as a type-level pairing function, is inherited from the language of NFU, and the constant symbol $\bar V$ will be used to obtain a notion of smallness. The terms of $\mathcal{L}^*$ are generated from $\bar V$, and both set and class variables using the function symbol $P$. The atomic formulae of $\mathcal{L}^*$ are all of the formulae in the form $s \in t$ and $s=t$ where $s$ and $t$ are $\mathcal{L}^*$-terms. In order to state the axioms of $S^*$ we first need to extend the notion of stratification to $\mathcal{L}^*$-formulae.

\begin{dfn}
An $\mathcal{L}^*$-formula $\phi$ is said to be stratified if the following conditions can be satisfied:
\begin{enumerate}
\item Each term $s$ in $\phi$ can be assigned a natural number that we will call the type assigned to $s$;
\item the type assigned to a term $s$ is the same as the type assigned to every class variable occurring in $s$; 
\item each class variable of $\phi$ has the same type assigned to all of its occurrences;
\item for each subformula of $\phi$ of the form $s=t$, the type assigned to the term $s$ is the same as the type assigned to the term $t$;
\item for each subformula of $\phi$ of the form $s \in t$, if $s$ is assigned the type $n$ then $t$ is assigned the type $n+1$.
\end{enumerate} 
\end{dfn}

The theory $S^*$ is the $\mathcal{L}^*$-theory that is axiomatised by the universal closures of the following:
\begin{enumerate}
\item (Stratified Comprehension) for all stratified $\mathcal{L}^*$-formulae $\phi(X, \vec{Z})$,
$$\exists Y \forall X(X \in Y \iff \phi(X, \vec{Z}))$$  
\item (Weak Extensionality)
$$\mathcal{S}(X) \land \mathcal{S}(Y) \Rightarrow (X=Y \iff \forall Z(Z \in X \iff Z \in Y))$$
\item (Pairing) $P(X_1, X_2)=P(Y_1, Y_2) \Rightarrow X_1=Y_1 \land X_2=Y_2$
\item (Sets and Classes)
\begin{enumerate}
\item $\forall x \exists X(x=X)$
\item $X \in \bar{V} \iff \exists x(x=X)$
\item $X \in x \Rightarrow X \in \bar{V}$
\end{enumerate}
\item (Empty Set) $\exists!z \forall y(y \notin z)$ --- we use $\varnothing$ to denote the unique $z \in \bar V$ such that $\forall y(y \notin z)$
\item (Operations on Sets)
\begin{enumerate}
\item $\{x, y\} \in \bar{V}$
\item $\bigcup x \in \bar{V}$
\item $\mathcal{P}(x) \in \bar{V}$
\item $P(x, y)= \{\{x\}, \{x, y\}\}$
\end{enumerate}
\item (Infinity)
$$\exists a(\exists z(z \in a \land \forall y(y \notin z)) \land \forall x(x \in a \Rightarrow x \cup \{x\} \in a))$$
\item (Replacement) for all $\mathcal{L}^*$-formulae $\phi(x, y, \vec{Z})$,
$$\forall x \forall y_1 \forall y_2(\phi(x, y_1, \vec{Z}) \land \phi(x, y_2, \vec{Z}) \Rightarrow y_1=y_2) \Rightarrow \forall a \exists b \forall y(y \in b \iff \exists x(x \in a \land \phi(x, y, \vec{Z})))$$
\item (Foundation) for all $\mathcal{L}^*$-formulae $\phi(x, \vec{Z})$,
$$\exists x \phi(x, \vec{Z}) \Rightarrow \exists x(\phi(x, \vec{Z}) \land (\forall y \in x) \neg \phi(y, \vec{Z}))$$
\item (Universal Choice)
$$\exists C\left( \begin{array}{c}
\forall X \forall Y_1 \forall Y_2(P(X, Y_1) \in C \land P(X, Y_2) \in C \Rightarrow Y_1 = Y_2) \land\\
\forall X(\exists Y(Y \in X) \Rightarrow \exists Y(Y \in X \land P(X, \{Y\}) \in C))
\end{array}\right)$$
\end{enumerate}  

Since $S^*$ has a two-sorted language with set- and class-variables, we will call arbitrary objects \em classes\em, and we will call elements of $\bar V$ \em sets\em. This contrasts with the practice in ZFC, where one talks of proper classes, but cannot prove their existence as objects. In $S^*$, just as in NBG or KMC, classes exist as objects. It follows immediately from the definition of $S^*$, as an extension of NFU, that it satisfies (R1) and (R2). Feferman proves that $S^*$ is consistent from ZFC + $\exists \kappa \exists \lambda$ ``$\kappa < \lambda$ are inaccessible cardinals'' \cite{Fef74} using a technique from \cite{Jen69} based on the Erd\H{o}s-Rado theorem. In this paper, the combinatorics of Feferman's proof is adjusted so as to prove the consistency of $S^*$ already from ZFC + $\exists \kappa$ ``$\kappa$ is an inaccessible cardinal''. 

$S^*$ satisfies (R3) in a stronger sense than NFU does, and such considerations motivate its strong replacement and foundation schemata. Let us work in $S^*$, interpreting ``small'' as ``$\in \bar V$'' and (provisionally) ``large'' as ``anything''. Let $\mathbf{V}$ denote the category of all sets with functions as morphisms. In \cite{Fef74}, Feferman proves a version of the Yoneda lemma from $S^*$: 

\begin{lemma}[Yoneda in $S^*$]
If all the objects and morphisms of $\mathbf{C}$ are small, and $\mathbf{F}$ is a functor from $\mathbf{C}$ to $\mathbf{V}$, then for each object $A$ of $\mathbf{C}$, there is a bijection $\mathrm{yon}_A : \mathbf{F}(A) \rightarrow \mathbf{Nat}(\mathbf{C}(A, -), \mathbf{F})$. 
\end{lemma}
$\mathbf{C}(A, -)$ denotes the covariant hom functor from $\mathbf{C}$ to $\mathbf{V}$ given by $A$, and $\mathbf{Nat}(\mathbf{C}(A, -), \mathbf{F})$ denotes the set of natural transformations from $\mathbf{C}(A, -)$ to $\mathbf{F}$. This differs in two ways from the standard Yoneda lemma. Firstly, we only require that the objects and morphisms are small, not the stronger condition that the category is locally small. Secondly, the co-domain of $\mathbf{F}$ is the category $\mathbf{V}$, not the category of small sets. But of course, local smallness of $\mathbf{C}$ is exactly what is needed for the hom functors to be small-valued, so this readily specializes to the standard statement.

This is a first step towards satisfying the category theory part of (R3), but it is a daunting task to verify all the standard results of category theory involving the large/small-distinction. We will instead indicate why we feel confident that almost all such results would go through. In the family of systems including some form of the stratified comprehension schema, the following notions turn out to be quite important.

\begin{dfn}
$A$ is \em cantorian \em if there is a bijection from $A$ to $\{\{a\} \mid a \in A\}$. $A$ is \em strongly cantorian \em if there is a bijection from $A$ to $\{\{a\} \mid a \in A\}$ that maps each $a \in A$ to $\{a\}$. We adopt $\stcan(A)$ as shorthand for ``$A$ is strongly cantorian''.
\end{dfn}

The following basic results are easily verified. $S^*$ proves that $\bar V$ is strongly cantorian. Moreover, NFU (and therefore $S^*$) proves that the powerset of a strongly cantorian set is strongly cantorian. NFU also proves that if $B$ is a set of bijections, each witnessing strong cantorianicity of its domain, then $\cup B$ is a bijection witnessing strong cantorianicity of its domain.

Working with the von Neumann ordinals $\mathrm{Ord}^{\bar V}$ in $\bar V$, and relying on the strong foundation axiom schema, we now wish to recursively define $\bar V_0 = \bar V$, $\bar V_{\alpha + 1} = \pow(\bar V_\alpha)$ and (if alpha is a limit ordinal) $\bar V_\alpha = \cup \{\bar V_\beta \mid \beta < \alpha\}$. Our intention is also to show by induction that $\bar V_\alpha$ is strongly cantorian, for any $\alpha \in \mathrm{Ord}^{\bar V}$. As we shall see, that will lend support to the claim that $S^*$ satisfies (R3) sufficiently. We will need the recursion theorem for NFU, which may be found here \cite[§15.1]{Hol98}. The foundation axiom schema of $S^*$ implies that $\Ord^{\bar V}$ forms a well-founded class. So we obtain the following special case of the recursion theorem for recursion over $\Ord^{\bar V}$.

\begin{thm}[Recursion in NFU and $S^*$]
Suppose that
\[\mathcal G : \{H : \alpha \rightarrow V \mid \alpha \in \mathrm{Ord}^{\bar V}\} \rightarrow \{\{x\} \mid x \in V\}.\]
Then there is a unique function $F : \Ord^{\bar V} \rightarrow V$, such that $\{F(\alpha)\} = \mathcal G(F \restriction \alpha)$, for each $\alpha \in \Ord^{\bar V}$.
\end{thm}

We now wish to apply this to obtain a hierarchy, recursively defined by $\bar V_0 = \bar V$, $\bar V_{\alpha + 1} = \pow(\bar V)$, and $\bar V_\alpha = \cup \{\bar V_\beta \mid \beta < \alpha\}$ (for $\alpha$ a limit). Note that in NFU, we define $\pow$ so that only the designated empty class, and no other atoms, are elements of a powerclass. The main obstacle to defining the required $\mathcal G$ of the recursion theorem, is that $\pow$ is not a function in NFU.

\begin{thm}
In $S^*$, there are unique functions $R, C : \Ord^{\bar V} \rightarrow V$, such that for all $\alpha \in \Ord^{\bar V}$
\[
\begin{array}{lcl}
R(\alpha) &=& \begin{cases}
\bar V & \textrm{if $\alpha = 0$} \\
\pow(R(\beta)) & \textrm{if $\alpha$ is a successor $\beta + 1$} \\ 
\cup \{R(\beta) \mid \beta < \alpha\} & \textrm{if $\alpha$ is a limit}
\end{cases}\\
C(\alpha) &=& \textrm{A bijection } R(\alpha) \rightarrow \{\{x\} \mid x \in R(\alpha)\} \textrm{ witnessing }\stcan(R(\alpha)).
\end{array}
\tag{\dag}
\]
\end{thm}
\begin{proof}
We need to restate $(\dag)$ in a stratified form. First note that stratified comprehension allows us to form the function 
\[\pow' : \{X \mid \forall Y \in X . \exists Z . Y = \{Z\}\} \rightarrow V,\]
defined by $\pow'(X) = \pow(\cup X)$, which is stratified. Let $\bar B$ be the bijection from $\bar V$ to $\{\{X\} \mid X \in \bar V\}$, such that $\forall X \in \bar V . \bar B(X) = \{X\}$. We can now reformulate $(\dag)$ as follows.
\[
\begin{array}{lcl}
R(\alpha) &=& \begin{cases}
\bar V & \textrm{if $\alpha = 0$} \\
\pow' \big( \{C(\beta)(X) \mid X \in R(\beta)\} \big) & \textrm{if $\alpha$ is a successor $\beta + 1$} \\ 
\cup \{R(\beta) \mid \beta < \alpha\} & \textrm{if $\alpha$ is a limit}
\end{cases}\\
C(\alpha) &=& \begin{cases}
\bar B & \textrm{if $\alpha = 0$} \\
\big\{ \langle S, \{\cup T\} \rangle \mid S \subset R(\beta) \wedge T = \{C(\beta)(X) \mid X \in S\} \big\} & \textrm{if $\alpha$ is a successor $\beta + 1$} \\
\cup \{C(\beta) \mid \beta < \alpha\} & \textrm{if $\alpha$ is a limit}
\end{cases}
\end{array}
\]
For applying the recursion theorem, we wish to show the existence of the function $F$ defined on $\Ord^{\bar V}$ by $\alpha \mapsto \langle R(\alpha), C(\alpha) \rangle$. The definition of the needed $\mathcal G$ is implicit from the expression above. $\mathcal G$ is obtained by comprehension if the expression is stratified, which is easily verified with this type-assignment:
\[
\begin{array}{lcl}
R, C, \pow' &\mapsto& 2 \\
R(\alpha), R(\beta), C(\alpha), C(\beta), S, T, \bar V &\mapsto& 1 \\
X, C(\beta)(X) &\mapsto& 0 \\
\end{array}
\]
Therefore, unique functions $R$ and $C$ exist, satisfying the reformulated version of $(\dag)$. But it remains to show that our reformulated version of $(\dag)$ is equivalent to $(\dag)$. I.e. it remains to clarify that $R(\alpha + 1) = \pow(R(\alpha))$, that $C(\alpha)(X) = \{X\}$ and that $\dom(C(\alpha)) = R(\alpha)$, for all $\alpha \in \Ord^{\bar V}$ and all $X \in \dom(C(\alpha))$. We do this through an induction argument, justified by the strong foundation axiom schema of $S^*$. So let $\alpha \in \Ord^{\bar V}$ and let $X \in \dom(C(\alpha))$.

Successor case, $\alpha = \beta + 1$: Assume that $C(\beta)(Y) = \{Y\}$, for all $Y \in \dom(C(\beta))$, and assume that $\dom(C(\beta)) = R(\beta)$. So by definition of $C$, we have that $X \subset R(\beta)$, and 
\[ 
C(\alpha)(X) = \big\{ \cup \{C(\beta)(Y) \mid Y \in X\} \big\} = \big\{ \cup \{\{Y\} \mid Y \in X\} \big\} = \{X\}. 
\]
Moreover, by definition of $R$,
\[
R(\alpha) = \pow' \big( \{C(\beta)(Y) \mid Y \in R(\beta)\} \big) = \pow(\cup \{\{Y\} \mid Y \in R(\beta)\}) = \pow(R(\beta)),
\]
and it follows that $\dom(C(\alpha)) = R(\alpha)$.

Limit case: Assume that $C(\beta)(Y) = \{Y\}$, and that $\dom(C(\beta)) = R(\beta)$, for all $\beta < \alpha$ and $Y \in \dom(C(\beta))$. By definition of $C$ and $R$, we have 
\[ 
\dom(C(\alpha)) = \cup \{\dom(C(\beta)) \mid \beta < \alpha\} = \cup \{R(\beta) \mid \beta < \alpha\} = R(\alpha),
\]
where the sets $\{\dom(C(\beta)) \mid \beta < \alpha\}$ and $\{R(\beta) \mid \beta < \alpha\}$ are obtained from stratified comprehension and the existence of $R$ and $C$. Therefore,
\[
\langle X, Z \rangle \in C(\alpha) \Leftrightarrow \exists \beta < \alpha . C(\beta)(X) = Z \Leftrightarrow Z = \{X\}.
\]
So $C$ and $R$ turn out to be the desired functions.
\end{proof}

Thanks to this result, we have the class $\bar V_\alpha := R(\alpha)$, for each $\alpha \in \Ord^{\bar V}$. By induction, these are all transitive classes, so it is easily seen that they are hereditarily strongly cantorian. Since we obtained a function $\alpha \mapsto \bar V_\alpha$ on $\Ord^{\bar V}$, we can take the union $\bar{\bar V} := \cup \{\bar V_\alpha \mid \alpha \in \Ord^{\bar V}\}$. We can also take the union $\bar{\bar B} := \cup \{C(\alpha) \mid \alpha \in \Ord^{\bar V}\}$, to obtain a bijection witnessing that $\bar{\bar V}$ is strongly cantorian.

\begin{thm}\label{ConZC}
In $S^*$, $\bar{\bar V}$ is a model of ZC $+$ ``$\Ord^{\bar V}$ is an inaccessible cardinal'' $+$ ``replacement schema for sets of cardinality in $\Ord^{\bar V}$''.\footnote{The proof given, for the separation and the weakened replacement schemata, quantifies over formulas in the meta-theory. So what is proved is actually a theorem schema about $\bar{\bar V}$, where $\bar{\bar V}$ is considered externally as a submodel of a model of $S^*$. However, Roland Hinnion's development of the semantics of first-order logic in \cite{Hin75} applies to NFU, and provides a means to internalize this proof to $S^*$.}
\end{thm}
\begin{proof}
$\Ord^{\bar V}$ is an inaccessible cardinal: It is clearly transitive and totally ordered by $\in$, and by foundation well-ordered by $\in$. Hence, it is an ordinal. Suppose towards a contradiction that it were not a cardinal. Then there would exist a bijection $f \in \bar V$ from $\kappa \in \Ord^{\bar V}$ to $\Ord^{\bar V}$. So by replacement, we would have $\Ord^{\bar V} \in \bar V$, which contradicts foundation. So $\Ord^{\bar V}$ is a cardinal. By the same argument (augmented with the powerset axiom), $2^\kappa$ does not stand in surjection to $\Ord^{\bar V}$, for any $\kappa \in \Ord^{\bar V}$. Lastly, suppose that there is $\kappa \in \Ord^{\bar V}$ and a function $f$ from $\kappa$ to $\Ord^{\bar V}$ that is unbounded. By replacement, its image $I$ is then an element of $\bar V$, and $\cap I = \Ord^{\bar V} \in \bar V$, a contradiction. We conclude that $\Ord^{\bar V}$ is an inaccessible cardinal.

Extensionality: None of the sets in $\bar{\bar V}$ contains any atom other than $\varnothing$.

Infinity: $\omega \in \bar V$.

Let $A, B \in \bar{\bar V}$. Then we may fix $\alpha + 2 \in \Ord^{\bar V}$, such that $A, B \in \bar V_{\alpha + 2}$. 

Union: Since $\bar V_{\alpha + 2} = \pow(\pow(\bar V_\alpha))$, we have $\cup A \subset \bar V_\alpha$, so $\cup A \in \bar V_{\alpha + 1}$.

Pair: $\{A, B \} \in \bar V_{\alpha + 3}$.

Powerset: Since $A \subset \bar V_{\alpha + 1}$, we have $\pow(A) \subset \pow(\bar V_{\alpha + 1})$, so $\pow(A) \in \bar V_{\alpha + 3}$.

Choice: Follows from global choice.

Foundation: Let $\beta$ be the least ordinal in $\Ord^{\bar V}$, such that there is $C \in \bar V_\beta \cap A$. For each $U \in C$, there is $\gamma < \beta$, such that $U \in \bar V_\gamma$. Hence, for each $U \in C$, we have that $U \not \in A$.

Separation schema: Let $\phi(X, Y)$ be a formula. We need to show that $\{X \in A \mid \phi^{\bar{\bar V}}(X, B)\}$ exists in $\bar{\bar V}$. Since $\bar{\bar V}$ and $B$ are strongly cantorian, the formula $X \in A \wedge \phi^{\bar{\bar V}}(X, B)$ is equivalent to a stratified formula. (This is done by using the bijections witnessing that the sets are strongly cantorian, to shift the types in the formula. For more details, see the Subversion theorem in \cite[§17.5]{Hol98}.) Hence, by stratified comprehension, we have that $\{X \in A \mid \phi^{\bar{\bar V}}(X, B)\}$ exists. Since it is a subset of $A$, it is a subset of $\bar V_{\alpha + 1}$, and therefore an element of $\bar V_{\alpha + 2}$.

Replacement schema for sets of cardinality in $\Ord^{\bar V}$: Suppose $|A| = \kappa \in \Ord^{\bar V}$. It suffices to show that if $\phi(\xi, Z, B)$ is a formula such that $[\forall \xi \in \kappa . \exists^! Z \phi(\xi, Z, B)]^{\bar{\bar V}}$, then there is $C \in \bar{\bar V}$ such that $[\forall Z ( Z \in C \leftrightarrow \exists \xi \in \kappa . \phi(\xi, Z, B))]^{\bar{\bar V}}$. Consider the formula $\psi(\xi, \zeta, B)$ defined as ``$\zeta$ is the least ordinal in $\Ord^{\bar V}$ such that there exists $Z \in \bar V_\zeta$ for which $\phi^{\bar{\bar V}}(\xi, Z, B)$''. By the strong replacement axiom schema of $S^*$, the image of $\kappa$ under $\psi$ exists in $\bar V$. Let $\lambda$ be the least ordinal in $\Ord^{\bar V}$, which is not in that image. Then we have $[\forall Z ( Z \in \bar V_\lambda \leftarrow \exists \xi \in \kappa . \phi(\xi, Z, B))]^{\bar{\bar V}}$. Now, by separation in $\bar{\bar V}$, we obtain $C := \{ Z \in \bar V_\lambda \mid [\exists \xi \in \kappa . \phi(\xi, Z, B)]^{\bar{\bar V}} \}$, as desired.
\end{proof}

We also extract the following general fact about ZC.

\begin{prop}
ZC $\vdash$ Replacement schema for functional formulas with bounded image.
\end{prop}
\begin{proof}
Suppose $\alpha \in \Ord$ and $\phi(X, Z, B)$ is a formula such that $\forall X \in A . \exists^! Z \phi(X, Z, B)$ and $\forall Z ( Z \in V_\alpha \leftarrow \exists X \in A . \phi(X, Z, B))$. Then, by separation, we obtain $C := \{ Z \in V_\alpha \mid \exists X \in A . \phi(X, Z, B) \}$, as desired.
\end{proof}

We now suggest to interpret ``small'' as ``$\in \bar V$'' and to re-interpret ``large'' as ``$\in \bar{\bar V}$''. We then have quite a workable set theory for the large categories. The kind of operations on large categories, that would require more than what we have in $\bar{\bar V}$, are unusual in ordinary mathematics and category theory. For example, we cannot form the union $\cup \{ \bar V_\alpha \mid \alpha \in \Ord^{\bar V} \}$ in the model $\bar{\bar V}$. On the other hand, if for some unusual application we would need to form that union, then we can make use of the fact that $\bar{\bar V}$ is strongly cantorian, and apply the iterated  powerclass argument above to obtain $\bar{\bar V}_0, \bar{\bar V}_1, \dots, \bar{\bar V}_\alpha, \dots$, as well as their limit, the model $\bar{\bar{\bar V}}$, in which $\cup \{ \bar V_\alpha \mid \alpha \in \Ord^{\bar V} \}$ is obtained. Of course, we can also iterate that argument. We conclude that $S^*$ is a level (2) solution to (R3), and thereby to all of (R).

We end this section with two remarks. Firstly, so far we only used the strong replacement axiom schema of $S^*$ in the proof of Theorem \ref{ConZC}, to show that $\bar{\bar V}$ is a model of ``the replacement schema for sets of cardinality in $\Ord^{\bar V}$''. Another benefit of the strong replacement schema for the foundations of category theory, is that the small-definability concerns encountered in ZFC/S disappear: Working in $S^*$, strong replacement immediately yields the following: If $f$ is a function and $A \in \bar V$ is a subset of $\dom(f)$, such that $f(A) \subset \bar V$, then $f(A) \in \bar V$. So for example, if $\mathbf{C}$ is a category, $A$ is a small subset of its set of objects, and for all $X, Y \in A$, the set $\mathbf{C}(X, Y)$ of all $\mathbf{C}$-morphisms from $X$ to $Y$ is small, then the natural restriction of $\mathbf{C}$ to $A$ is small. Without strong replacement, this argument requires the additional small-definability assumption that the natural restriction of $\mathbf{C}$ to $\bar V$ is definable over $\bar V$, which is not necessarily the case.

Secondly, on this interpretation of small and large, the categories enabling us to satisfy (R1) and (R2), are neither large nor small. The notions of small and large have come to be associated with various results of category theory, that do not necessarily hold for those categories. We will look at one example of this phenomenon shortly, but first let us consider the theory NFUA as a foundation of category theory.

\subsection{NFUA as a solution to both (R) and (S)}

NFUA is defined as NFU + ``every cantorian set is strongly cantorian''. NFUA satisfies (R1) and (R2) simply because NFU does. In an unpublished proof, Robert Solovay showed in 1995 that NFUA is equiconsistent with ZFC + the schema ``for each $n \in \mathbb{N}$, there is an $n$-Mahlo cardinal''. A cardinal $\kappa$ is (defined as) $0$-Mahlo, if the set of inaccessible cardinals below $\kappa$ is stationary below $\kappa$. And $\kappa$ is $n+1$-Mahlo, if the set of $n$-Mahlo cardinals below $\kappa$ is stationary below $\kappa$, where $n \in \mathbb{N}$. A set $S \subset \kappa$ is (defined as) stationary below $\kappa$, if $S$ intersects every subset of $\kappa$ that is unbounded and closed under suprema below $\kappa$. The proof of a refinement of Solovay's equiconsistency result is presented in \cite{Ena04}.

For each $n \in \mathbb{N}$, the theory ZFC + ``there is an $n$-Mahlo cardinal'' is interpretable in NFUA by means of the set of equivalence classes of pointed extensional well-founded structures. This technique, invented by Roland Hinnion \cite{Hin75}, is often used to obtain lower bounds on the consistency strength of variants of NFU and NF. We may of course solve (S) within this interpretation, for example through the Grothendieck approach of subsection \ref{ApproachesInaccessibles}. In the terminology of subsection \ref{FormulationsRandS}, that approach yields a level (3) solution to (R3) and (S1). The proof of the converse direction of the equiconsistency result shows that we can add a relatively consistent axiom to NFUA that provides us with a level (2) implementation. We will briefly explain how this works for a certain refined solution to (S), that is similarly motivated as ZFC/S.

Just below ZFC + $\exists \kappa$ ``$\kappa$ is a Mahlo cardinal'', in consistency strength, we have the theory ZMC defined roughly as ZFC + ``$|V|$ is a Mahlo-cardinal''. Of course, we cannot state the cardinality of $V$ since it is not a set, but we can add the schema ``for each formula $\phi(\xi)$, such that the class of ordinals satisfying $\phi$ is unbounded and closed under suprema, there is an inaccessible cardinal $\lambda$ satisfying $\phi$''. This schema essentially says that $|V|$ is a Mahlo cardinal.

As explained by Shulman \cite{Shu08}, we can form ZMC/S, in analogy with Feferman's ZFC/S. The same proof as with ZFC/S, gives us that ZMC/S is a conservative extension of ZMC. It turns out that ZMC/S is equivalent to ZFC/S + ``$|\mathbb{S}|$ is inaccessible''. This allows us to obtain a stronger replacement axiom schema in relation to $\mathbb{S}$: If $A \in \mathbb{S}$ and $\phi(x, y, u)$ is a formula with $u$ an arbitrary parameter (i.e. not necessarily an element of $\mathbb{S}$), such that $\forall x \in A . \exists^! y \in \mathbb S . \phi(x, y, u)$, then there is $B \in \mathbb{S}$, such that $\forall y . (y \in B \leftrightarrow \exists x \in A . \phi(x, y, u))$. This means that we do not need to worry about small-definability. Hence, ZMC/S is quite a convenient solution to (S).

To interpret ZMC/S in NFUA, directly with the element relation of the language, restricted to a set, we may proceed as follows. In \cite{Ena04} a model of NFUA is constructed from a countable model of the theory $T = $ ZFC + the schema ``for each $n \in \mathbb{N}$, there is an $n$-Mahlo cardinal''. Performing the ``over S'' transformation to $T$ (just as from ZFC to ZFC/S in subsection \ref{FefermanZFCoverS}), gives us the theory $T$/S. As in the proof of Theorem \ref{ConZFCoverS}, Con$(T) \Rightarrow$ Con$(T$/S$)$. So we obtain a countable model of $T$/S. Then, since $T \subset T$/S, the construction of the model $\mathcal M$ of NFUA goes through as before. In $\mathcal M$, ZMC/S is directly interpretable by restricting the element relation of the language to a set, yielding a level (2) solution to (R3) and (S1).

NFUA solves (R) and (S) somewhat separately: (S) may be solved within the interpretation of ZMC/S in NFUA (this also yields (R3)), while (R1) and (R2) are solved by means of categories (typically built from the universal set $V$) that do not fit within the interpretation of ZMC/S. For someone who is only interested in foundations solving (S), it would obviously be more economical to work in ZMC/S than in NFUA. So it is natural to wonder: Is there a benefit for category theory in enabling the existence of categories yielding (R1) and (R2), apart from the intuitive appeal of obtaining e.g. the category of all groups without any small/large restrictions? It seems that little research has been directly aimed at answering this question. In the next subsection we will just scratch the surface, by expressing and proving a basic fact of stratified set theory in category theoretic form. This result contrast with Theorem \ref{FreydLimits}, showing that proper (R1)-categories can, in a sense, be more complete than what is possible in a purely ZFC-based setting.

\subsection{The limits of a proper category (in NFU)}

We already saw that in ZFC-based foundations, a category $\mathbf{C}$ which is not a preorder, can only have products indexed by a set of lower cardinality than the cardinality of the set of morphisms of $\mathbf{C}$. The proof in the ZFC setting uses Cantor's theorem, which holds for cantorian sets in NFU. But since $V \supset \pow(V) $, the universe $V$ is a counterexample to Cantor's theorem, and the categories yielding (R1) will typically have sets of objects and morphisms derived from $V$. Cantor's theorem has the following type-shifted version in NFU, essentially proved the same way.

\begin{thm}
$\mathrm{NFU} \vdash \forall x . \big|\big\{\{y\} \mid y \in x\big\}\big| < |\pow(x)|$.
\end{thm}

Again observing that $\pow(V) \subset V$, we now see that $\big|\big\{\{y\} \mid y \in V\big\}\big| < |V|$. As an example of how categories in NFU satisfying (R1) can behave quite differently from their counterparts in ZFC, we will now show that the category $\mathbf{Rel}$, with all sets as objects and binary relations as morphisms, has products and coproducts indexed by the set of all singletons. Similarly, the category $\mathbf{Set}$ has coproducts indexed by the set of all singletons. Compare this with the fact (a consequence of Theorem \ref{FreydLimits}) that the corresponding locally small categories in a ZFC-based setting, do not have such limits indexed by the set/class of all small singletons. Note that in order to accommodate the weaker extensionality axiom of NFU, we stipulate that neither an object nor a morphism of $\mathbf{Rel}$ or $\mathbf{Set}$ may be an atom other than $\varnothing$.

Before getting into this proof, we would like to take the opportunity to point out that Thomas Forster, Adam Lewicki and Alice Vidrine have done unpublished work on the category theory of sets in stratifiable set theories at Cambridge University. This work establishes several category theoretic properties of NF: For example, the category $\mathbf{Set}$ in NF is exhibited as an ``almost topos'' and as a category of classes. Moreover, the status of the Yoneda lemma internal to NF is clarified. Being a category of classes entails, among other properties, having finite limits and finite coproducts. In the present paper we are simply being explicit about how large the index category can be; there is no technical novelty in the proof. The point is to show possibilities in NFU-based category theory that are not available in purely ZFC-based approaches.

\begin{prop} \label{RelProducts}
$\mathrm{NFU} \vdash \mathbf{Rel}$ has products indexed by $\big\{\{i\} \mid i \in V \big\}$.
\end{prop}
\begin{proof}
Let $\mathbf{Sing}$ be the category with $\big\{\{i\} \mid i \in V \big\}$ as its set of objects and only identity morphisms, and let $\mathbf{F} : \mathbf{Sing} \rightarrow \mathbf{Rel}$ be a diagram. Set 
\[P := \big\{ \langle x, i \rangle \mid x \in \mathbf{F}(\{i\}) \big\},\] 
and define $\pi_{\{i\}} : P \rightarrow _\mathbf{Rel} \mathbf{F}(\{i\})$ by
\[\pi_{\{i\}} := \{\langle \langle x, i \rangle, x \rangle \mid x \in \mathbf{F}(\{i\}) \},\] 
for each $\{i\} \in \mathbf{Sing}$. Note that the definition of the map $\{i\} \mapsto \pi_{\{i\}}$ is stratified, and therefore realized by a function.

Consider an arbitrary cone to $\mathbf{F}$ in $\mathbf{Rel}$, i.e. a set $A$ and a relation $R_{\{i\}} : A \rightarrow_\mathbf{Rel} \mathbf{F}(\{i\})$, for each $\{i\} \in \mathbf{Sing}$. Again, the map-definition $\{i\} \mapsto R_{\{i\}}$ is stratified and therefore realized by a function. We may now define $u : A \rightarrow P$ by 
\[u := \{\langle a, \langle x, i \rangle \rangle \mid a R_{\{i\}} x\}.\] 
We just need to check that for an arbitrary $\{i\} \in \mathbf{Sing}$, we have $R_{\{i\}} = \pi_{\{i\}} \circ u$, i.e. that
\begin{center}
\includegraphics{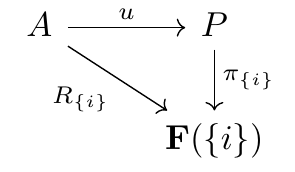}
\end{center}
commutes. So suppose $a \in A$ and $x \in \mathbf{F}(\{i\})$. Then, 
\[ a (\pi_{\{i\}} \circ u) x \iff \exists p \in P . (a u p \wedge p \pi_{\{i\}} x) \iff a u \langle x, i \rangle \iff a R_{\{i\}} x .\]
Since a morphism of $\mathbf{Rel}$ cannot be any other atom than $\varnothing$, we have by extensionality that $R_{\{i\}} = \pi_{\{i\}} \circ u$.
\end{proof}

For any relation $R : A \rightarrow B$, we have a converse relation $R^\dag : B \rightarrow A$ defined by $x R y \Leftrightarrow y R^\dag x$. ($\mathbf{Rel}$ is a dagger symmetric monoidal category, hence the notation.) Clearly, $(Q \circ R)^\dag = R^\dag \circ Q^\dag$, for any morphisms $Q, R$ in $\mathbf{Rel}$. Thus, by replacing $u$, the $\pi_{\{i\}}$ and the $R_{\{i\}}$ by their converses, in the proof above, we see:

\begin{prop}\label{RelCoproducts}
$\mathrm{NFU} \vdash \mathbf{Rel}$ has coproducts indexed by $\big\{\{i\} \mid i \in V \big\}$.
\end{prop}

Note that the $\pi_{\{i\}}^\dag$ are functions. Moreover, if we assume that the $R_{\{i\}}^\dag$ are functions, then it follows that $u^\dag$ is a function. Therefore:

\begin{prop}\label{SetCoproducts}
$\mathrm{NFU} \vdash \mathbf{Set}$ has coproducts indexed by $\big\{\{i\} \mid i \in V \big\}$.
\end{prop}

Since we did not use choice, these results hold in $\mathrm{NFU} \setminus \{\mathrm{AC}\}$ and its extension NF (with full extensionality) as well. 

As explained above, the proof given for Theorem \ref{FreydLimits} in ZFC, is blocked in NFU. But since $\big|\big\{\{y\} \mid y \in V\big\}\big| < |V|$, these propositions are not counterexamples to Theorem \ref{FreydLimits} in NFU. So the question whether NFU proves Theorem \ref{FreydLimits} or not, remains open, at least modulo the present paper. The point of Propositions \ref{RelProducts} -- \ref{SetCoproducts} is that we have positive category theoretic results in NFU (assuming that limits are generally desirable), whose obvious translations into ZFC-based foundations are falsified by Theorem \ref{FreydLimits}. We take this to be an indication that the (R1)-categories in NFU-based foundations may be of value to category theory, thus strengthening the motivation for exploring category theory internal to NF and NFU, as well as extensions like $S^*$ and NFUA. 

\section{The consistency of $S^*$}

%{\bf [Zach: How about finishing off the paper with your section 3, i.e. the proof of Con(S*)?]}%

In this section we present Feferman's consistency proof of his system $S^*$. By carefully keeping track of the cardinals used to build a model of $S^*$ we are able to improve on the result reported in \cite{Fef74, Fef06} by showing that $\mathrm{ZFC}+ \exists \kappa \textrm{``$\kappa$ is an inaccessible cardinal''}$ proves the consistency of $S^*$. 

We will make reference to two subsystems of ZFC studied in \cite{Mat01}. Mac Lane Set Theory ($\mathrm{Mac}$) is the subsystem of ZFC axiomatised by: extensionality, pair, emptyset, union, powerset, infinity, $\Delta_0$-separation, transitive containment, regularity and the axiom of choice. The set theory Kripke-Platek with Ranks ($\mathrm{KPR}$) is obtained from $\mathrm{Mac}$ by deleting the axiom of choice and adding $\Delta_0$-collection, $\Pi_1$-foundation and an axiom asserting that for every ordinal $\alpha$, the set $V_\alpha$ exists \footnote{Mathias omits powerset and transitive containment from his axiomatisation of $\mathrm{KPR}$ in \cite{Mat01}, but both of these axioms follow from the existence of $V_\alpha$ for every ordinal $\alpha$.}. If $\mathcal{L}^\prime$ is an extension of the language of set theory and $\mathcal{M}$ is an $\mathcal{L}^\prime$-structure then for $a \in M$ we write $a^*$ for the class $\{x \in M \mid \mathcal{M} \models (x \in a)\}$.   

Throughout this section we work in the theory $\mathrm{ZFC}+\exists \kappa$ ``$\kappa$ is an inaccessible cardinal''. Let $\kappa$ be an inaccessible cardinal. Let $A= V_\kappa \cup \{V_\kappa\}$. Jensen's consistency proof of NFU \cite{Jen69} reveals that a model of NFU can be built from model $\mathcal{M}$ of Mac that admits a non-trivial automorphism $j$ and such that there exists a point $c \in M$ with
\begin{equation} \label{eq:AutConditionForNFU}
\mathcal{M} \models (c \cup \mathcal{P}(c) \subseteq j(c)).
\end{equation}
In \cite{Fef74} Feferman obtains a model of $S^*$ (in the theory $\mathrm{ZFC} + \exists \kappa \exists \lambda$ ``$\kappa < \lambda$ are inaccessible cardinals'') by building a model $\mathcal{M}$ of ZFC that admits a non-trivial automorphism (the existence of the rank function $\alpha \mapsto V_\alpha$ in ZFC ensures that every non-trivial automorphism satisfies (\ref{eq:AutConditionForNFU})) and such that $\mathcal{M}$ is an end-extension of $A$. This ensures that there is an isomorphic copy $A$ in the well-founded part of the resulting model $\mathcal{N}$ of NFU, and moreover every point in this isomorphic copy of $A$ is strongly Cantorian in $\mathcal{N}$. This allows Feferman to interpret the constant symbol $\bar V$ using the point corresponding to $V_\kappa$ in $\mathcal{N}$. The fact that $\bar V$ is isomorphic to $V_\kappa$ where $\kappa$ is an inaccessible cardinal and $\mathcal{N}$ exists as a set ensures that $\bar V$ satisfies axioms (5--9) of $S^*$.

Feferman builds a model of ZFC that end-extends $A$ and admits a non-trivial automorphism using tools from infinitary logic. The fine-tuned proof we present here uses the same techniques as \cite{Fef74}, however the model $\mathcal{M}$ that we build, which is an end-extension of $A$ and admits a non-trivial automorphism, will satisfy a fragment of ZFC that is sufficient to ensure that $\mathcal{M}$ equipped with its non-trivial automorphism still yields a model of NFU. Let $\mathcal{L}^A$ be the extension of the language of set theory obtained by adding new constant symbols $\hat{a}$ for every $a \in A$. We use $\mathcal{L}_{\infty \omega}^A$ to denote the infinitary language obtained from $\mathcal{L}^A$ that permits arbitrarily long conjunctions and disjunctions, but only finite blocks of quantifiers. Let $T_A$ be the $\mathcal{L}_{\infty \omega}^A$-theory with axioms:
$$\forall x \left(x \in \hat{a} \iff \bigvee_{b \in a} (x=\hat{b})\right) \textrm{ for each } a \in A.$$
Note that $T_A$ asserts that an $\mathcal{L}^A$-structure is an end-extension of $A$. Let $\mathcal{L}_A$ be the least Skolem fragment of $\mathcal{L}_{\infty \omega}^A$ that contains $T_A$. So, in addition to the symbols of $\mathcal{L}^A$, $\mathcal{L}_A$ contains an $n$-ary function symbol $F_{\exists x \phi}$ for each $n$-ary $\mathcal{L}_A$-formula $\exists x \phi(x, y_1, \ldots, y_n)$. We write $\mathrm{Sk}_A$ for the $\mathcal{L}_A$-theory with axioms:
$$\forall \vec{y}(\exists x \phi(x, \vec{y}) \Rightarrow \phi(F_{\exists x \phi}(\vec{y}), \vec{y})) \textrm{ for each } \mathcal{L}_A \textrm{-formula } \phi(x, \vec{y}).$$
Let $\mathrm{Fm}_A$ be the set of $\mathcal{L}_A$-formulae. Note that $|\mathrm{Fm}_A|= \kappa$ and so there $2^\kappa$ many $\mathcal{L}_A$-theories.

We obtain a model that admits a non-trivial automorphism from a model equipped with an infinite class of order indiscernibles.

\begin{dfn}
Let $\mathcal{M}$ be an $\mathcal{L}_A$-structure. We say that a linear order $\langle I, < \rangle$ with $I \subseteq M$ is a set of $n$-variable indiscernibles for $\mathcal{M}$ if for all $\mathcal{L}_A$-formulae $\phi(x_1, \ldots, x_n)$ and for all $a_1 < \cdots < a_n$ and $b_1 < \cdots < b_n$ in $I$,
$$\mathcal{M} \models \phi(a_1, \ldots, a_n) \textrm{ if and only if } \mathcal{M} \models \phi(b_1, \ldots, b_n).$$
If for all $n \in \omega$, $\langle I, < \rangle$ with $I \subseteq M$ is a set of $n$-variable indiscernibles for $\mathcal{M}$ then we say that $\langle I, < \rangle$ is a set of indiscernibles for $\mathcal{M}$.
\end{dfn}

The following result from \cite{BK71} allows us to build $\mathcal{L}_A$-structures with indiscernibles:

\begin{lemma} \label{Th:InfinitaryLogicLemma}
(Barwise-Kunen) Suppose that for all $n \in \omega$, $\mathcal{M}_n$ is model of $\mathrm{Sk}_A$ and $\langle I_n, <_n \rangle$ with $I_n \subseteq M_n$ is a set of $n$-variable indiscernibles for $\mathcal{M}_n$. If for all $n \in \omega$, for all $a_1 <_n \cdots <_n a_n$ in $I_n$ and for all $b_1 <_{n+1} \cdots <_{n+1} b_n$ in $I_{n+1}$,
$$\langle \mathcal{M}_n, a_1, \ldots, a_n \rangle \equiv_{\mathcal{L}_A} \langle \mathcal{M}_{n+1}, b_1, \ldots, b_n \rangle$$
then for any linear order $\langle I, < \rangle$ there is an $\mathcal{L}_A$-structure $\mathcal{M}$ with $I \subseteq M$ such that $I$ is a set of indiscernibles for $\mathcal{M}$ and for all $n \in \omega$, for all $a_1 <_n \cdots <_n a_n$ in $I_n$ and for all $b_1 < \cdots < b_n$ in $I$,
$$\langle \mathcal{M}_n, a_1, \ldots, a_n \rangle \equiv_{\mathcal{L}_A} \langle \mathcal{M}, b_1, \ldots, b_n \rangle.$$
\end{lemma}

We now turn to building a model of $T_A$ with indiscernibles. The following definition appears in \cite{Jen69}.

\begin{dfn}
Let $I$ be a set and let $\langle f_n \mid n \in \omega \rangle$ be a family of partitions such that for all $n \in \omega$, $f_n$ has domain $[I]^n$. We say that $\langle c_n \mid n \in \omega \rangle$ is realizable for $\langle f_n \mid n \in \omega \rangle$ if for all $n \in \omega$ and for all $\beta < |I|$, there exists an $I_n \subseteq I$ such that $|I_n| \geq \beta$ and $f_k([I_n]^k)=\{c_k\}$ for all $k \leq n$.
\end{dfn}

The $n$-variable indiscernibles on the $\mathcal{L}_A$-structures required so that we can apply Lemma \ref{Th:InfinitaryLogicLemma} are obtained by considering a sequence of partitions $\langle f_n \mid n \in \omega \rangle$ that admit a realizable sequence $\langle c_n \mid n \in \omega \rangle$. If $\lambda$ is a cardinal then the generalized beth operation is defined by recursion: $\beth_0(\lambda)= \lambda$, $\beth_{\alpha+1}(\lambda)= 2^{\beth_\alpha(\lambda)}$ and (if $\alpha$ is a limit ordinal) $\beth_\alpha(\lambda)= \sup \{ \beth_\beta(\lambda) \mid \beta < \alpha\}$. If $\lambda$ is a cardinal then the function $\alpha \mapsto \beth_\alpha(\lambda)$ is inflationary and continuous. It follows from this observation that the function $\alpha \mapsto \beth_\alpha(\lambda)$ has arbitrarily large fixed points. We will use the Erd\H{o}s-Rado Theorem \cite{ER56} to produce realizable sequences.

\begin{lemma} \label{Th:ErdosRado}
(Erd\H{o}s-Rado) Let $\lambda$ be an infinite cardinal. If $f: [X]^{n+1} \longrightarrow \lambda$ is a partition with $|X| \geq \beth_n(\lambda)^+$ then there exists $H \subseteq X$ with $|H| \geq \lambda^+$ and $\gamma \in \lambda$ such that $f``[H]^{n+1}= \{\gamma\}$.
\end{lemma}

We now turn to building a model of $\mathrm{KPR}$ that end-extends $A$ and admits a non-trivial automorphism. Let $\eta_0$ be the least ordinal such that $\beth_{\eta_0}(2^\kappa)= \eta_0$. Now, recursively define
$$\eta_{\alpha+1} \textrm{ to be the least ordinal }>\eta_\alpha \textrm{ such that } \beth_{\eta_{\alpha+1}}(2^\kappa)= \eta_{\alpha+1},$$
$$\eta_\alpha= \sup_{\beta < \alpha} \eta_\beta \textrm{ for limit } \alpha.$$
Since the function $\alpha \mapsto \beth_\alpha(2^\kappa)$ is continuous it follows that for all ordinals $\alpha$,
$$\eta_\alpha= \beth_{\eta_\alpha}(2^\kappa).$$ 
Therefore, for all ordinals $\alpha$, $V_{\eta_\alpha}= H_{\eta_\alpha}$ and so $\langle V_{\eta_\alpha}, \in \rangle$ is a model of $\mathrm{KPR}$. Let
$$I= \{V_{\eta_\alpha} \mid \alpha < \eta_{(2^\kappa)^+} \}.$$
The membership relation ($\in$) linearly orders $I$ and we will often abuse notation and use $<$ to denote this linear order. Now, $|I|= \eta_{(2^\kappa)^+}$ and $\mathrm{cf}(\eta_{(2^\kappa)^+})= (2^\kappa)^+ > 2^\kappa$. The argument used to prove the following Lemma can be found in \cite{Jen69}.

\begin{lemma} \label{Th:RealizableSequenceLemma}
If $\langle f_n \mid n \in \omega \rangle$ is a family of partitions such that for all $n \in \omega$,
$$f_n: [I]^{n+1} \longrightarrow 2^\kappa$$
then there exists a sequence $\langle c_n \mid n \in \omega \rangle$ that is realizable for $\langle f_n \mid n \in \omega \rangle$. 
\end{lemma}

\begin{proof}
We inductively construct $\langle c_n \mid n \in \omega \rangle$. Suppose that we have $c_0, \ldots, c_{n-1} \in 2^\kappa$ such that for all $\beta < |I|$, there exists $D \subseteq I$ with $|D| \geq \beta$ and
$$f_k([D]^{k+1})= \{c_k\} \textrm{ for all } k < n.$$
Let $\mathcal{U}_{n-1}$ be the set of all $D \subseteq I$ such that
$$f_k([D]^{k+1})= \{c_k\} \textrm{ for all } k < n.$$
Note that our inductive hypothesis ensures that $\mathcal{U}_{n-1}$ contains arbitrarily large subsets of $I$. Now, let $\mathcal{U}^\prime$ be the set of all $B \subseteq I$ such that there exists a $D \in \mathcal{U}_{n-1}$ and $c \in 2^\kappa$ with $B \subseteq D$ and
$$f_n([B]^{n+1})= \{c\}.$$
Lemma \ref{Th:ErdosRado} ensures that $\mathcal{U}^\prime$ contains arbitrarily large subsets of $I$. Using the fact that $\mathrm{cf}(|I|)> 2^\kappa$ we can find $c_n \in 2^\kappa$ and $\mathcal{U}_n \subseteq \mathcal{U}^\prime$ which contains arbitrarily large subsets of $I$ such that for all $B \in \mathcal{U}_n$,
$$f_n([B]^{n+1})=\{c_n\}.$$
Therefore, for all $\beta < |I|$, there exists $D \subseteq I$ with $|D| \geq \beta$ and 
$$f_k([D]^{k+1})= \{c_k\} \textrm{ for all } k \leq n.$$
Therefore we can build $\langle c_n \mid n \in \omega \rangle$ that is realizable for $\langle f_n \mid n \in \omega \rangle$ by induction.       
\end{proof}

Let $\lhd$ be a well-ordering of $V_{\eta_{\eta_{(2^\kappa)^+}}}$. Define an $\mathcal{L}_A$-structure 
$$\mathcal{M}= \langle M, \in, (\hat{a}^\mathcal{M})_{a \in A}, (F_{\exists x \phi}^\mathcal{M})_{\exists x \phi \in \mathcal{L}_A} \rangle$$ 
such that:
\begin{itemize}
\item $M= V_{\eta_{\eta_{(2^\kappa)^+}}}$,
\item for all $a \in A$, $\hat{a}^\mathcal{M}= a$,
\item for all $\mathcal{L}_A$-formulae $\exists x \phi(x, \vec{y})$,
$$F_{\exists x \phi}^\mathcal{M}(\vec{y})= \left\{ \begin{array}{ll}
\varnothing & \textrm{if } \mathcal{M} \models \neg \exists x \phi(x, \vec{y})\\
\lhd\textrm{-least } a \textrm{ s.t. } \mathcal{M} \models \phi(a, \vec{y}) & \textrm{otherwise}
\end{array}\right.$$
\end{itemize}
Therefore, we have
$$\mathcal{M} \models (\mathrm{KPR}+\mathrm{Sk}_A+T_A) \textrm{ and } I \subseteq M.$$

\begin{lemma} \label{Th:IndiscernibleSequenceLemma}
There exists a family $\langle I_n \mid n \in \omega \rangle$ of infinite subsets of $I$ such that for all $n \in \omega$, $I_n$ is a set of $n$-variable indiscernibles for $\mathcal{M}$, and for all $a_0 < \cdots < a_n$ in $I_n$ and for all $b_0 < \cdots < b_n$ in $I_{n+1}$,
$$\langle \mathcal{M}, a_0, \ldots, a_n \rangle \equiv_{\mathcal{L}_A} \langle \mathcal{M}, b_0, \ldots, b_n \rangle.$$
\end{lemma}

\begin{proof}
For each $n \in \omega$, define $f_n:[I]^{n+1} \longrightarrow \mathcal{P}(\mathrm{Fm}_A)$ by
$$f_n(\{a_0 < \cdots < a_n\})= \{\phi \mid \phi \textrm{ is an } n+1\textrm{-ary } \mathcal{L}_A\textrm{-formula and } \mathcal{M} \models \phi(a_0, \ldots, a_n)\}.$$
Using Lemma \ref{Th:RealizableSequenceLemma} we can find a sequence $\langle c_n \mid n \in \omega \rangle$ that is realizable for $\langle f_n \mid n \in \omega \rangle$. Therefore we can inductively build a sequence $\langle I_n \mid n \in \omega \rangle$ such that for all $n \in \omega$, $|I_n| \geq \omega$ and 
$$f_k([I_n]^{k+1})= \{c_k\} \textrm{ for all } k \leq n.$$
Let $n \in \omega$. It is clear that $I_n$ is a set of $n$-variable indiscernibles for $\mathcal{M}$. Let $a_0 < \cdots < a_n$ be in $I_n$ and let $b_0 < \cdots < b_n$ be in $I_{n+1}$. We have
$$f_n(\{a_0 < \cdots < a_n\})= c_n= f_n(\{b_0 < \cdots < b_n\}).$$
Therefore, for all $\mathcal{L}_A$ formulae $\phi(x_0, \ldots, x_n)$,
$$\mathcal{M} \models \phi(a_0, \ldots, a_n) \textrm{ if and only if } \mathcal{M} \models \phi(b_0, \ldots, b_n).$$
Therefore
$$\langle \mathcal{M}, a_0, \ldots, a_n \rangle \equiv_{\mathcal{L}_A} \langle \mathcal{M}, b_0, \ldots, b_n \rangle.$$
\end{proof}

Let $\langle I_n \mid n \in \omega \rangle$ be the sequence whose existence is guaranteed by Lemma \ref{Th:IndiscernibleSequenceLemma}. We can now apply Lemma \ref{Th:InfinitaryLogicLemma} to obtain an $\mathcal{L}_A$-structure $\mathcal{M}^\prime$ and a linear order $\langle Q, < \rangle$ with $Q \subseteq M^\prime$ satisfying
\begin{itemize}
\item[(I)] the order-type of $\langle Q, < \rangle$ is $\mathbb{Z}$--- we write $Q= \{q_i \mid i \in \mathbb{Z}\}$, 
\item[(II)] $Q$ is a set of indiscernibles for $\mathcal{M}^\prime$,
\item[(III)] $\mathcal{M}^\prime \equiv_{\mathcal{L}_A} \mathcal{M}$,
\item[(IV)] for all $\mathcal{L}_A$-formulae $\phi(x_0, \ldots, x_n)$, if 
$$\mathcal{M} \models \phi(a_0, \ldots, a_n) \textrm{ for all } a_0 < \cdots < a_n \textrm{ in } I \textrm{ then}$$
$$\mathcal{M}^\prime \models \phi(q_{i_0}, \ldots, q_{i_n}) \textrm{ for all } i_0 < \cdots < i_n \textrm{ in } \mathbb{Z}.$$
\end{itemize}
Let $\mathcal{M}^{\prime\prime}$ be the $\mathcal{L}_A$-substructure of $\mathcal{M}^\prime$ generated by $Q$, and the constants and functions of $\mathcal{L}_A$. Therefore $\mathcal{M}^{\prime\prime} \prec_{\mathcal{L}_A} \mathcal{M}^\prime$. Consider the order automorphism $j^\prime: Q \longrightarrow Q$ defined by $j^\prime(q_i)= q_{i+1}$ for all $i \in \mathbb{Z}$. Since every element of $\mathcal{M}^{\prime\prime}$ is the result of applying an $\mathcal{L}_A$-Skolem function to a finite tuple of elements of $Q$, the bijection $j^\prime$ can be raised to an automorphism $j: \mathcal{M}^{\prime\prime} \longrightarrow \mathcal{M}^{\prime\prime}$. Note that for all $a < b$ in $I$,
$$\mathcal{M} \models (a \cup \mathcal{P}(a) \subseteq b).$$
Therefore, it follows from (IV) above that
\begin{equation} \label{eq:NFUAutCondition}
\mathcal{M}^{\prime\prime} \models (q_0 \cup \mathcal{P}(q_0) \subseteq j(q_0)).
\end{equation}
Define the $\mathcal{L}^*$-structure $\mathcal{N}= \langle N, N_{\mathrm{sets}}, \in^\mathcal{N}, P^\mathcal{N}, \bar{V}^\mathcal{N} \rangle$ by
\begin{itemize}
\item $N= q_0^*$,
\item $N_{\mathrm{sets}}= (\hat{V}_\kappa^{\mathcal{M}^{\prime\prime}})^*$, 
\item for all $x, y \in q_0^*$,
$$x \in^\mathcal{N} y \textrm{ if and only if } \mathcal{M}^{\prime\prime} \models (j(y) \subseteq q_0 \land x \in j(y)),$$
$$P^\mathcal{N}(x, y)= \{\{x\}, \{x, y \}\} \in q_0^*,$$
\item $\bar V^\mathcal{N}= \hat{V}_\kappa^{\mathcal{M}^{\prime\prime}}$
\end{itemize}  
where the set variables $x, y, z, \ldots$ range over the domain $N_{\mathrm{sets}}$ and the class variables $X, Y, Z, \ldots$ range over the domain $N$.

\begin{lemma}
$\mathcal{N} \models S^*$.
\end{lemma}

\begin{proof}
Note that since every element of $N_{\mathrm{sets}}$ interprets a constant symbol in $\mathcal{M}^{\prime\prime}$, the automorphism $j$ fixes every element of $N_{\mathrm{sets}}$. This means that the structure $\langle \bar V^*, \in^\mathcal{N} \rangle$ is isomorphic to $\langle V_\kappa, \in \rangle$. The arguments in \cite{Jen69} show that since $q_0$ satisfies (\ref{eq:NFUAutCondition}) and $j$ fixes every element of $N_{\mathrm{sets}}$, $\mathcal{N}$ satisfies axiom 1 and axiom scheme 2 of $S^*$. The fact that Universal Choice is an $\mathcal{L}^*$-stratified sentence that holds in every element of $I$ implies axiom 10 of $S^*$ holds in $\mathcal{N}$. The fact that axiom 3 of $S^*$ holds in $\mathcal{N}$ follows immediately from the definition of $P^\mathcal{N}$. The fact that axioms 4-7 and axiom scheme 9 of $S^*$ hold in $\mathcal{N}$ follow immediately from the fact that  $\langle \bar V^*, \in^\mathcal{N} \rangle$ is isomorphic to $\langle V_\kappa, \in \rangle$. Since the structure $\mathcal{N}$ is a set (in the metatheory) and $\langle \bar V^*, \in^\mathcal{N} \rangle$ is isomorphic to $\langle V_\kappa, \in \rangle$ it follows that axiom scheme 8 of $S^*$ holds in $\mathcal{N}$.        
\end{proof}

Since the structure $\mathcal{N}$ is a set in the theory $\mathrm{ZFC}+\exists \kappa$ ``$\kappa$ is an inaccessible cardinal'' we have shown:

\begin{thm}
$\mathrm{ZFC}+\exists \kappa$ ``$\kappa$ is an inaccessible cardinal'' $\vdash \mathrm{Con}(S^*)$.
\end{thm}

In section \ref{Sec:FefSstar} we showed that $S^*$ proves the existence of a set model, $\bar{\bar{V}}$, of $\mathrm{ZC}+\exists \kappa$ ``$\kappa$ is an inaccessible cardinal''. This shows:

\begin{thm}
$S^* \vdash \mathrm{Con}(\mathrm{ZC}+\exists \kappa$ ``$\kappa$ is an inaccessible cardinal''$)$.
\end{thm}

\begin{qu} \label{Q:ConsistencySrengthOfSstar}
What is the exact consistency strength of $S^*$ relative to an extension of ZFC?
\end{qu}

Motivated by Question \ref{Q:ConsistencySrengthOfSstar} Feferman and the first author of this paper proposed an extension of $S^*$ which is called $S^{**}$ in \cite{Fef06}. The theory $S^{**}$ is obtained from $S^*$ by adding the universal closure of the following axiom scheme:
\begin{itemize}
\item[11.] for all $\mathcal{L}^*$-formulae $\phi(x, \vec{Z})$,
$$\exists X (\forall x \in \bar V)(x \in X \iff \phi(x, \vec{Z})).$$ 
\end{itemize}
We will conclude this section by sketching how our proof that the theory $\mathrm{ZFC}+\exists \kappa$ ``$\kappa$ is an inaccessible cardinal'' proves the consistency of $S^*$ can be modified to show that the theory $\mathrm{ZFC}+\exists \kappa$ ``$\kappa$ is an inaccessible cardinal'' also proves the consistency of $S^{**}$. Work in the theory $\mathrm{ZFC}+\exists \kappa$ ``$\kappa$ is an inaccessible cardinal'' and let $\kappa$ be and inaccessible cardinal. Let $A^\prime= V_{\kappa+1}$. We can then define a fragment of infinitary logic $\mathcal{L}_{A^\prime}$ that is equipped with Skolem functions, and capable of asserting that a structure is an end-extension of $A^\prime$, with the property that $|\mathrm{Fm}_{A^\prime}|= 2^\kappa$. Therefore there are $2^{2^\kappa}= \beth_2(\kappa)$ many $\mathcal{L}_{A^\prime}$-theories. We now define
$$I= \{ V_{\eta_\alpha} \mid \alpha < \eta_{\beth_2(\kappa)^+}\}.$$
Therefore $|I|= \eta_{\beth_2(\kappa)^+}$ and $\mathrm{cf}(|I|)= \beth_2(\kappa)^+ > \beth_2(\kappa)$. We can then modify the definition of the structure $\mathcal{M}$ above by setting $M= V_{\eta_{\eta_{\beth_2(\kappa)^+}}}$ and expanding the interpretations to all symbols in $\mathcal{L}_{A^\prime}$. The resulting $\mathcal{L}_{A^\prime}$-structure is an end-extension of $A^\prime$, and satisfies $\mathrm{KPR}$ and the Skolem theory of $\mathcal{L}_{A^\prime}$. Using the same argument that we used above we can use $\mathcal{M}$ to build an $\mathcal{L}_{A^\prime}$-elementary equivalent structure $\mathcal{M}^{\prime\prime}$ that admits a non-trivial automorphism $j: \mathcal{M}^{\prime\prime} \longrightarrow \mathcal{M}^{\prime\prime}$. We define an $\mathcal{L}^*$-structure $\mathcal{N}$ from $\mathcal{M}^{\prime\prime}$ and $j$ in the same way as we did above. The structure $\mathcal{N}$ satisfies all of the axioms of $S^*$. Since $\mathcal{M}^{\prime\prime}$ is an end-extension of $A^\prime$, the structure $\langle \mathcal{P}(\bar V)^*, \in^\mathcal{N} \rangle$ will be isomorphic to $\langle V_{\kappa+1}, \in \rangle$. This is enough to ensure that axiom scheme 11 holds in $\mathcal{N}$. Thus we have:

\begin{thm}
$\mathrm{ZFC}+\exists \kappa$ ``$\kappa$ is an inaccessible cardinal'' $\vdash \mathrm{Con}(S^{**})$.
\end{thm}

\begin{qu}
What is the exact consistency strength of of $S^{**}$ relative to an extension of ZFC?
\end{qu}

\end{document}